\documentclass[12pt]{amsart}
\usepackage[applemac]{inputenc}
\RequirePackage{etex}

\usepackage[colorlinks=true,linkcolor=magenta,citecolor=blue,urlcolor=blue,hypertexnames=false,linktocpage]{hyperref}
\usepackage{bookmark}
\usepackage{amsthm,thmtools,amssymb,amsmath,amscd,amsfonts}
\usepackage{color,soul}
\usepackage{cleveref}
\usepackage{fancyhdr}
\usepackage{esint}
\usepackage{enumerate}
\usepackage{pictexwd,dcpic}
\usepackage{caption}
\declaretheorem[name=Theorem,numberwithin=section]{thm}
\declaretheorem[name=Remark,style=remark,sibling=thm]{rem}
\declaretheorem[name=Lemma,sibling=thm]{lemma}

\declaretheorem[name=Definition,style=definition,sibling=thm]{defn}
\declaretheorem[name=Corollary,sibling=thm]{cor}
\declaretheorem[name=Assumption,style=definition,sibling=thm]{assum}

\newtheorem{ot}{Theorem}

\numberwithin{equation}{section}

\usepackage{cleveref}
\crefname{lemma}{Lemma}{Lemmata}
\crefname{prop}{Proposition}{Propositions}
\crefname{thm}{Theorem}{Theorems}
\crefname{cor}{Corollary}{Corollaries}
\crefname{defn}{Definition}{Definitions}
\crefname{example}{Example}{Examples}
\crefname{rem}{Remark}{Remarks}
\crefname{assum}{Assumption}{Assumptions}
\crefname{nota}{Notation}{Notation}

\usepackage{autonum}

\newcommand{\ti}{\tilde}

\newcommand{\cn}{\colon}
\newcommand{\sub}{\subset}

\newcommand{\bbR}{\mathbb{R}}

\newcommand{\bbS}{\mathbb{S}}
\newcommand{\bbH}{\mathbb{H}}

\newcommand{\8}{\infty}

\newcommand{\al}{\alpha}
\newcommand{\be}{\beta}
\newcommand{\ga}{\gamma}
\newcommand{\de}{\delta}
\newcommand{\ep}{\varepsilon}
\newcommand{\ka}{\kappa}
\newcommand{\la}{\lambda}
\newcommand{\om}{\omega}
\newcommand{\si}{\sigma}
\newcommand{\Si}{\Sigma}

\newcommand{\Om}{\Omega}

\newcommand{\Ga}{\Gamma}

\newcommand{\La}{\Lambda}

\newcommand{\cU}{\mathcal{U}}

\newcommand{\del}{\partial}
\newcommand{\n}{\nabla}

\newcommand{\fa}{\forall}

\newcommand{\fr}[2]{\frac{#1}{#2}}
\newcommand{\x}{\times}

\DeclareMathOperator{\linspan}{span}

\DeclareMathOperator{\id}{id}

\DeclareMathOperator{\tr}{tr}
\DeclareMathOperator{\Rm}{Rm}

\newcommand{\pf}[1]{\begin{proof}#1 \end{proof}}
\newcommand{\eq}[1]{\begin{equation}\begin{alignedat}{2} #1 \end{alignedat}\end{equation}}

\newcommand{\br}[1]{\left(#1\right)}
\newcommand{\abs}[1]{\lvert #1\rvert}
\newcommand{\enum}[1]{\begin{enumerate}[(i)] #1 \end{enumerate}}

\newcommand{\ra}{\rightarrow}
\newcommand{\hra}{\hookrightarrow}

\newcommand{\mt}{\mapsto}

\newcommand{\mrm}{\mathrm}

\newcommand{\hp}{\hphantom}
\newcommand{\q}{\quad}

\usepackage{slashed}
\usepackage{accents}

\newcommand{\Dm}{\ensuremath{D_m}}

\begin{document}
	
	\title[A Viscosity Approach to CRT]{Constant Rank Theorems for Curvature Problems via a Viscosity approach}
	\author[P. Bryan, M. N. Ivaki, J. Scheuer]{Paul Bryan, Mohammad N. Ivaki, Julian Scheuer}
	\dedicatory{}
	\date{\today}
	\subjclass[2020]{}
	\begin{abstract}
		An important set of theorems in geometric analysis consists of constant rank theorems for a wide variety of curvature problems. In this paper, for geometric curvature problems in compact and non-compact settings, we provide new proofs which are both elementary and short. Moreover, we employ our method to obtain constant rank theorems for homogeneous and non-homogeneous curvature equations in new geometric settings. One of the essential ingredients for our method is a generalization of a differential inequality in a viscosity sense satisfied by the smallest eigenvalue of a linear map (Brendle-Choi-Daskalopoulos, Acta Math. 219(2017): 1--16) to the one for the subtrace. The viscosity approach provides a concise way to work around the well known technical hurdle that eigenvalues are only Lipschitz in general. This paves the way for a simple induction argument.
	\end{abstract}
	\maketitle
	
	\section{Introduction}
	We introduce a viscosity approach to a broad class of constant rank theorems. Such theorems say that under suitable conditions a positive semi-definite bilinear form on a manifold, that satisfies a uniformly elliptic PDE, must have constant rank in the manifold. In this sense, constant rank theorems can be viewed as a strong maximum principle for tensors. The aim of this paper is two-fold. Firstly, we want to present a new approach to constant rank theorems. It is based on the idea that the subtraces of a linear map satisfy a linear differential inequality in a viscosity sense and the latter allows to use the strong maximum principle. This avoids the use of nonlinear test functions, as in \cite{Bian2009}, as well as the need for approximation by simple eigenvalues, as in \cite{Szekelyhidi2016}. Secondly, we show that the simplicity of this method allows us to obtain previously undiscovered constant rank theorems, in particular for non-homogeneous curvature type equations.
	To illustrate the idea, we give a new proof for the following full rank theorem for the Christoffel-Minkowski problem, a.k.a. the $\sigma_{k}$-equation.
	\begin{thm}\cite[Thm.~1.2]{Guan2003a}
		Let $(\mathbb{S}^n,g,\nabla)$ be the unit sphere with standard round metric and connection.
		Suppose $n\geq 2$, $1\leq k\leq n-1$ and $0<s,\phi\in C^{\infty}(\mathbb{S}^n)$ satisfy
		\eq{\n^{2}\phi^{-\fr 1k}+\phi^{-\fr 1k}g\geq 0,\; 0\leq r:=\n^{2}s+sg\in \Ga_{k},\; \sigma_{k}(r)= \phi,}
		where $\si_{k}$ is $k$-th symmetric polynomial of eigenvalues of $r$ with respect to $g$ and $\Ga_{k}$ is the $k$-th Garding cone. Then $r$ is positive definite.
	\end{thm}
	\pf{
		For convenience, we define \eq{F=\sigma_{k}^{1/k},\; f=\phi^{1/k}.} Then $F=f.$
		Differentiate $F$ and use Codazzi, where a semi-colon stands for covariant derivatives and we use the summation convention:
		\eq{f_{;ab}=F_{;ab}&=F^{ij,kl}r_{ij;a}r_{kl;b}+F^{ij}r_{ij;ab}\\
			&=F^{ij,kl}r_{ij;a}r_{kl;b}+F^{ij}r_{ab;ij}-r_{ab}F^{ij}g_{ij}+Fg_{ab}.}
		Hence the tensor $r$ satisfies the elliptic equation
		\eq{\label{ell}F^{ij}r_{ab;ij}=F^{ij}g_{ij}r_{ab}-fg_{ab}-F^{ij,kl}r_{ij;a}r_{kl;b}+f_{;ab}.}
		Now we deduce an inequality for the lowest eigenvalue of $r$, $ \la_{1}$, in a viscosity sense. Let $\xi$ be a smooth lower support at $x_0\in \bbS^{n}$ for $\la_{1}$ and let $D_1\geq 1$ denote the multiplicity of $\la_{1}(x_0)$. Denote by $\La$ the complement of the set $\{i,j,k,l>D_1\}$ in $\{1,\dots,n\}^{4}$. We use a relation between the derivatives of $\xi$ and $r$, and the inverse concavity of $F$ (cf. \cite[Lem.~5]{BrendleChoiDaskalopoulos:/2017}, \cite{Andrews:/2007}) to estimate in normal coordinates at $x_0$:
		\eq{F^{ij}\xi_{;ij}&\leq F^{ij} r_{11;ij}-2\sum_{j>D_1}\fr{F^{ii}}{\la_{j}}(r_{ij;1})^{2}\\
			&=-F^{ij,kl}r_{ij;1}r_{kl;1}-2\sum_{j>D_1}\fr{F^{ii}}{\la_{j}}(r_{ij;1})^{2}+F^{ij}g_{ij}r_{11}-(f-f_{;11})\\
			&=-\sum_{i,j,k,l>D_1}F^{ij,kl}r_{ij;1}r_{kl;1}-2\sum_{j>D_1}\fr{F^{ii}}{\la_{j}}(r_{ij;1})^{2}-(f-f_{;11})\\
			&\hp{=}~-\sum_{(i,j,k,l)\in \La}F^{ij,kl}r_{ij;1}r_{kl;1}+F^{ij}g_{ij}r_{11}\\
			&\leq-\br{f+2f^{-1}f_{;1}^2-f_{;11}} + c\abs{\n\xi}+F^{ij}g_{ij}\xi\\
			&\leq F^{ij}g_{ij}\xi+ c\abs{\n\xi}.}
		Then the strong maximum principle for viscosity solutions (cf. \cite{Bardi1999}) implies that the set $\{\la_{1}=0\}$ is open. Hence, if $\la_{1}$ was zero somewhere, it would be zero everywhere. However, we know it is positive somewhere, since at a minimum of $s$ we have $r>0$.
	}
	
	The proof may be summarized as follows: apply the viscosity differential inequality from \cite[Lem.~5]{BrendleChoiDaskalopoulos:/2017} for the minimum eigenvalue \(\lambda_1\) of the spherical hessian of $r$. Then the strong maximum principle shows that since there is a point at which \(\lambda_1 > 0\) we must have \(\lambda_1 > 0\) everywhere and hence the hessian has constant, full rank. A similar argument was employed in \cite{Ivaki2019} for obtaining curvature estimates along a curvature flow. 
	
	Our main approach here is to generalize the viscosity inequality to the subtrace $G_m = \lambda_1 + \dots + \lambda_m$, the sum of the first $m$ eigenvalues. See \Cref{BCD-extension lemma} below. Then by induction, we are able to show that if $\lambda_1 = \dots = \lambda_{m-1} \equiv 0$, the strong maximum principle shows that either $G_m > 0$ or $G_m \equiv 0$ to conclude constant rank theorems (in short, CRT). 
	
	We say a symmetric 2-tensor $\alpha$ is Codazzi, provided $\nabla \alpha$ is totally symmetric.	Here is a prototypical CRT:
	
	\begin{thm}[Homogeneous CRT]\label{Codazzi version}\cite[Thm.~1.4]{Caffarelli2007}
		Suppose $\alpha$ is a Codazzi, non-negative, symmetric $2$-tensor on a connected Riemannian manifold $(M,g,\nabla)$ satisfying $\Psi(\alpha,g) = f> 0$, where $\Psi$ is one-homogeneous, inverse concave  and strictly elliptic (see \Cref{Def:F} and \Cref{assumption p-CM}), and we have $\nabla^2 f^{-1} + \tau f^{-1} g \geq 0$ with $\tau(x)$ the minimum sectional curvature at $x$. Then $\alpha$ is of constant rank.\footnote{Note that in \cite[Thm.~1.4]{Caffarelli2007}, $F:=-\Psi^{-1}$.}
	\end{thm}

	We state a more general version of CRT that allows the curvature function to be non-homogeneous and to explicitly depend on $x\in M$ as well. To state the result, we need a few definitions.

	\begin{defn}\label{Def:F}
		Let $\Ga\sub\bbR^{n}$ be an open, convex cone such that 
		\eq{\Ga_{+} := \{\la\in \bbR^{n}\cn \la_{i}>0\q\fa 1\leq i\leq n\} \sub\Ga.}
		Suppose $(M^n,g)$ is a smooth Riemannian manifold. A $C^{\8}$-function
		\eq{F\cn \Ga\x M\ra \bbR}
		is said to be a  {\it{pointwise curvature function}}, if for any $x\in M$, the map $F(\cdot,x)$ is symmetric under permutation of the $\la_{i}$. Such a map generates another  map (denoted by $F$ again) given by
		\eq{F\cn \cU\sub \bbR^{n\x n}_{\mathrm{sym}}\x \bbR^{n\x n}_{\mrm{sym}}\x M&\ra \bbR\\									(\al,g,x)&\mt F(\al,g,x) = F(\la,x),}
		where $\cU$ is a suitable open set and $\la = (\la_{i})_{1\leq i\leq n}$ are the eigenvalues of $\al$ with respect to $g$, or equivalently, the eigenvalues of the linear map $\al^{\sharp}$ defined by
		$g(\al^{\sharp}(v),w) = \al(v,w).$ Note that  $F$ can be considered as a map on an open set of $\bbR^{n\x n}$ via $F(\al^{\sharp},x) = F(\al,g,x)$;
		see \cite{Scheuer:06/2018}. 
		
		With the convention $\al^{i}_{j} = g^{ik}\al_{kj}$, where $(g^{kl})$ is the inverse of $(g_{kl})$:
		\eq{F^{i}_{j} := \fr{\del F}{\del \al^{j}_{i}},\q F^{ij} := \fr{\del F}{\del \al_{ij}},\q F^{ij,kl} := \fr{\del F}{\del \al_{ij}\del \al_{kl}}.}
		Note that $F^{ij} = F^{i}_{k}g^{kj}$. Moreover,  $F$ is said to be 
		\enum{
			\item {\it{strictly elliptic}}, if 
			$F^{ij}\eta_{i}\eta_{j} >0\q\fa 0\leq \eta\in 
			\bbR^{n},$
			\item {\it{one-homogeneous}}, if for all $x\in M$, $F(\cdot,x)$ is homogeneous of degree one, and
			\item {\it{inverse concave}}, if the map $\ti F \in C^{\8}(\Ga_{+}\x M)$ defined by
			\eq{\ti F(\la_{i},x)= - F(\la_{i}^{-1},x)\q \mbox{is concave.} }
		}
	\end{defn}
	
	We use the convention for the Riemann tensor from \cite{Gerhardt:/2006}. For a Riemannian or Lorentzian manifold $(M,g,\nabla)$,
	\eq{\Rm(X, Y) Z: = \nabla_Y \nabla_X Z - \nabla_X \nabla_Y Z - \nabla_{[Y,X]} Z}
	and we lower the upper index to the first slot:
	\eq{\Rm(W, X, Y, Z) = g(W, \Rm(X, Y) Z).}
	The respective local coordinate expressions are $(R^{m}_{jkl})$ and $(R_{ijkl})$.

	\begin{defn}\label{Phi def}~
		\enum{
			\item A pointwise curvature function $F\in C^{\8}(\Ga\x M)$ is 
			$\Phi$-inverse concave for some 
			\eq{\Phi\in C^{\8}(\Ga\x M, T^{4,0}(M)),}
			provided at all $\be>0$ we have
			\eq{F^{ij,kl}\eta_{ij}\eta_{kl}+2F^{ik}\ti\be^{jl}\eta_{ij}\eta_{kl}\geq \Phi^{ij,kl}\eta_{ij}\eta_{kl},}
			where $\ti\be^{ik}\be_{kj}=\de^{i}_{j}$.
			\item 	
			For $\al\in \Ga$ we define a curvature-adjusted modulus of $\Phi$-inverse concavity,
			\eq{\om_{F}(\al)(\eta,v)&=\Phi^{ij,kl}\eta_{ij}\eta_{kl}+D_{xx}^{2}F(v,v)+2D_{x^{k}}F^{ij}\eta_{ij}v^{k}\\
				&\hp{=~}+ \tr_{g}\Rm(\al^{\sharp},v,D_{\al^{\sharp}}F,v),}
                        where  $D$ denotes the product connection on $\bbR^{n\x n}\x M$. Here the curvature term denotes contracting the vector parts of the $(1,1)$ tensors $\al^{\sharp} = \al^i_j , D_{\al^{\sharp}} F = F^k_l$ with the Riemann tensor and tracing the resulting bilinear form with respect to the metric so that
                        \eq{\tr_{g}\Rm(\al^{\sharp},e_m,D_{\al^{\sharp}}F,e_m) = g^{jl} \al^i_j F^k_l R_{imkm}.}
		}
		
	\end{defn}

        \begin{rem}
        If $(A, x) \mapsto -F(A^{-1}, x)$ is concave (i.e., $F$ is inverse concave), then we take $\Phi=0$ and for all $(\eta,v)$ we have 
        \eq{\om_F(\eta,v)\geq \tr_g \Rm(\al^{\sharp},v,D_{\al^{\sharp}}F,v).}
        On several occasions, where there is a homogeneity condition on $F$, we will be able to choose a good positive $\Phi$ that allows to relax assumptions on the other variables of the operator $F$; see \Cref{sec:results}.
        \end{rem}

	We state the main result of the paper which contains \Cref{Codazzi version} as a special case. 
	\begin{thm}[Non-homogeneous CRT]\label{CRT}
		Let $(M,g,\nabla)$ be a connected Riemannian manifold and $\Ga$ an open, convex cone containing $\Ga_{+}$. Suppose $F\in C^{\8}(\Ga\x M)$ is a $\Phi$-inverse concave, strictly elliptic pointwise curvature function. Let $\alpha$ be a Codazzi, non-negative, symmetric $2$-tensor with eigenvalues in $\Ga$ and
		\eq{F(\alpha^{\sharp},\cdot) = 0\quad\mbox{on}~M.}
		Suppose for all $\Om\Subset M$ there exists a positive constant $c=c(\Om)$, such that for all eigenvectors $v$ of $\al^{\sharp}$ there holds
		\eq{\om_{F}(\al)(\n_{v}\al,v)\geq -c(\al(v,v)+\abs{\n\al(v,v)}).}
		Then $\alpha$ is of constant rank. 
	\end{thm}

	\begin{rem}
        It might seem more natural to replace the condition on $\om_F$ with the condition
        \eq{\om_{F}(\al)(\eta, v) \geq -c(\al(v,v)+\abs{\n\al(v,v)})}
        for every $\eta$ and all $v$. Indeed such a condition certainly leads to constant rank theorems since taking in particular $\eta = \nabla_v \al$, and $v$ and eigenvector, we may apply \Cref{CRT}. However, the requirement holding for all $\eta, v$ is too restrictive for applications such as in \Cref{Codazzi version}. See the proof in \Cref{sec:results} below where the required inequality is only proved to hold for $\eta = \nabla_v \al$ and $v$ an eigenvector.
        \end{rem}
	%

	An application of \Cref{CRT} to a non-homogeneous curvature problem is given in \Cref{non-homogeneous-appl}. Such a result was declared interesting in \cite{Guan2019a}. The full results are listed in \Cref{sec:results}.  
	

	CRT (also known as the microscopic convexity principle) was initially developed in \cite{MR792181} in two-dimensions for convex solutions of semi-linear equations, $\Delta u = f(u)$ using the maximum principle and the homotopy deformation lemma. The result was extended to higher dimensions in \cite{MR856307}. The continuity method combined with a CRT yields existence of strictly convex solutions to important curvature problems. For example, a CRT was an important ingredient in the study of prescribed curvature problems such as the Christoffel-Mink\-owski problem and prescribed Weingarten curvature problem \cite{Guan2003a,Guan2006a,GuanLinMa:/2006}. Later, general theorems for fully nonlinear equations were obtained in \cite{Caffarelli2007,Bian2009} under the assumption that $A \mapsto F(A^{-1})$ is locally convex.
		These approaches are based on the observation that a non-negative definite matrix valued function $A$ has constant rank if and only if there is a $\ell$ such that the elementary symmetric functions satisfy $\sigma_{\ell} \equiv 0$ and $\sigma_{\ell-1} > 0$. To apply this observation requires rather delicate, long computations and the introduction of clever auxiliary functions. The difficulties are at least in part due to the non-linearity of $\sigma_{\ell}$. An alternative approach was taken in \cite{Szekelyhidi2016,Szekelyhidi2020}, using a linear combination of lowest $m$ eigenvalues, which provides a linearity advantage at the expense of losing regularity compared with $\sigma_{\ell}$. The authors get around this difficulty by perturbing $A$ so that the eigenvalues are distinct (thus restoring regularity) but then using an approximation argument. Our approach based on the viscosity inequality shows that $G_m$ enjoys sufficient regularity to apply the strong maximum principle and this suffices to obtain a self-contained proof of the CRT.
		
 We remark here, that our method is capable of reproving the results in \cite{Caffarelli2007,Bian2009}, namely with the help of \Cref{viscosity} it is possible to prove that any convex solution $u$ to
	\eq{H(\n^{2}u,\n u,u,\cdot) = 0}
has constant rank under the assumption that 
\eq{(A,u,x)\mt -H(A^{-1},p,u,x)}
is concave for fixed $p$. This result does not follow from \Cref{CRT}, but by using a suitably redefined $\om_{F}$ in \Cref{viscosity}, this result follows in the same way as \Cref{CRT}. Here we rather want to focus on geometric problems.
	
We proceed as follows: In \Cref{sec:results} we collect and prove direct applications of \Cref{CRT}. In \Cref{visc} we prove the viscosity inequality satisfied by the subtrace, a result that is of interest by itself. After some further corollaries, we conclude with the proof of \Cref{CRT}.

	\section{Applications}
	\label{sec:results}
	In this section, we collect a few applications of \Cref{CRT}. We fix an assumption that we need on several occasions.

	\begin{assum}\label{assumption p-CM}
		Let $\Ga$ be as in \Cref{Def:F}.
		\enum{
			\item $\Psi\in C^{\8}(\Gamma)$ is a positive, strictly elliptic, homogeneous function of degree one and normalized to
			$\Psi(1,\dots,1)=n,$
			\item $\Psi$ is inverse concave.
		}
	\end{assum}
	
	Recall that such a function $\Psi$
	at invertible arguments $\be$ satisfies
	\eq{\label{inverse concavity}
		\Psi^{ij,kl}\eta_{ij}\eta_{kl}+2\Psi^{ik}\ti\be^{jl}\eta_{ij}\eta_{kl}\geq \fr{2}{\Psi}(\Psi^{ij}\eta_{ij})^{2}}
	for all symmetric $(\eta_{ij})$; see for example \cite{Andrews:/2007}.
	
	In order to facilitate notation, for covariant derivatives we use semi-colons, e.g., the components of the second derivative $\n^{2}T$ of a tensor are denoted by
	\eq{T_{;ij}=\n_{\del_{j}}\n_{\del_{i}}T-\n_{\n_{\del_{j}\del_{i}}}T.}
	
	First, we illustrate how \Cref{Codazzi version} follows from \Cref{CRT}.
	
	\begin{proof}[Proof of \Cref{Codazzi version}] 
		We define $F = \Psi - f.$ In view of \eqref{inverse concavity} and \Cref{Phi def}, we have
		\[\Phi^{ij,kl}\eta_{ij}\eta_{kl}=2\Psi^{-1}(\Psi^{ij}\eta_{ij})^{2}.\]
		
		Let $x_0\in M$ and $(e_i)_{1\leq i\leq n}$  be an orthonormal basis of eigenvectors for $\al^{\sharp}(x_0)$. In the associated coordinates, we calculate
		\eq{\omega_F(\al)( \n_{e_{m}}\al,e_m)&\geq 2f^{-1}f_{;m}^2 -f_{;mm}+\tau\Psi^{kr}\al^l_k(g_{lr} - g_{lm}g_{rm})\\
			&\geq 2f^{-1}f_{;m}^2 - f_{;mm}+\tau f - c\al_{mm}\\
			&=f^{2}\br{(f^{-1})_{;mm}+\tau f^{-1} } - c\al_{mm},
		}
		for some constant $c.$ Hence the claim follows from \Cref{CRT}.
	\end{proof}
	
	For a $C^2$ function $\zeta$ on a space $(M,g)$ of constant curvature $\tau_M$,
	\eq{\label{r}r_M[\zeta]:=\tau_M\nabla^2\zeta+ g\zeta.}

	The next theorem contains the full rank theorems from \cite{Guan2003a,Hu2004,Guan2006a} as special cases.
	\begin{thm}[$L_p$-Christoffel-Minkowski Type Equations]\label{FRT for p-CM}
		Suppose $(M,g,\nabla)$ is either the hyperbolic space $\bbH^{n}$ or the sphere $\bbS^{n}$ equipped with their standard metrics and connections.
		Let $\Psi$ satisfy \Cref{assumption p-CM}, $k\geq 1,$ $p\neq 0$ and $0<\phi,s\in C^{\infty}(M)$ satisfy
		\eq{r_{M}[s]&\geq 0,\; s^{1-p}\Psi^k(r_M[s])= \phi.}
		If either
		\eq{
			\left\{
			\begin{array}{ll}
				r_{\bbH^n}[\phi^{-\frac{1}{p+k-1}}]\geq 0, & p+k-1<0, \\
				\mbox{or}\\
				r_{\bbS^n}[\phi^{-\frac{1}{p+k-1}}]\geq 0, & p\geq 1,
			\end{array}
			\right.
		}
		then $r_M[s]$ is of constant rank. In particular, if $M=\bbS^n,$ then we have \eq{r_{\bbS^n}[s]>0.}
	\end{thm}
	\begin{proof}
		Note that $\alpha=r_M[s]$  is a Codazzi tensor. We define
		\[F=\Psi -\left(\phi s^{p-1}\right)^{\frac{1}{k}}=\Psi-f.\]
		For simplicity, we rewrite $f=us^{q-1},$ where  $u=\phi^{\frac{1}{k}}$  and $q=\frac{p+k-1}{k}.$

		As in the proof of \Cref{Codazzi version}, we have
		\eq{\omega_F(\al)( \n_{e_{m}}\al,e_m)\geq
			2f^{-1}f_{;m}^2 - f_{;mm}+\tau_M f  - c\al_{mm}.
		}
		Now we calculate
		\eq{
			f_{;mm}-2f^{-1}f_{;m}^2-\tau_Mf&=-\left(\tau_Mqu+\frac{q+1}{q}\frac{(u_{;m})^2}{u}-u_{;mm}\right)s^{q-1}\\
			&\hp{=~}-\frac{q-1}{q}\left(\frac{u_{;m}}{u}+q\frac{s_{;m}}{s}\right)^2f\\
			&\hp{=~}+\tau_M(q-1)fs^{-1}r_M[s]_{mm}.
		}
		Therefore, if either $r_{\bbH^n}[u^{-\frac{1}{q}}]\geq 0,\; q<0$ or $r_{\bbS^n}[u^{-\frac{1}{q}}]\geq 0,\; q\geq 1,$ then
		\eq{f_{;mm}-2f^{-1}f_{;m}^2-\tau_Mf\leq c\al_{mm},}
		for some $c\geq 0.$ The result follows from \Cref{CRT}. Since $\bbS^n$ is compact, at some point $y$ we must have $r_{\mathbb{S}^n}[s](y)>0.$ Hence $r_{\mathbb{S}^n}[s]>0$ on $M.$ 	
	\end{proof}
	\begin{rem}
		Let $M=x(\Omega)$, $x:\Omega \hookrightarrow \bbR^{n,1}$ be a co-compact, convex, spacelike hypersurface. The support function of $M$, $s:\bbH^n\to \bbR$, is defined by
		$s(z)=\inf\{-\langle z,p\rangle;\, p\in M\},$ 
		and $r_{\bbH^n}[s]$ is non-negative definite. Moreover, if $r>0$, then the eigenvalues of $r$ with respect to $g$ are the principal radii of curvature; e.g., \cite{AndrewsChenFangMcCoy:/2015}. Therefore, the curvature problem stated in the previous theorem can be considered as an $L_p$-Christoffel-Minkowski type problem in the Minkowski space.
	\end{rem}
	
	In \cite{Guan2019a} the authors asked the validity of CRT for non-homogeneous curvature problems. In this respect we have the following theorem. First we have to recall the definition of the {\it{Garding cones}}:
	
	\eq{\Ga_{\ell} = \{\la\in \bbR^{n}\cn \si_{1}(\la)>0,\dots,\si_{\ell}(\la)>0\},}
	where $\si_{k}$ is the $k$-th elementary symmetric polynomial of the $\la_{i}$. In $\Ga_{\ell}$, all $\si_{k}$, $1\leq k\leq \ell$, are strictly elliptic and the $\si_{k}^{1/k}$ are inverse concave, see \cite{HuiskenSinestrari:09/1999}. For a cone $\Ga\sub\bbR^{n}$, on a Riemannian manifold $(M,g)$ a bilinear form $\al$ is called {\it{$\Ga$-admissible}}, if its eigenvalues with respect to $g$ are in $\Ga$.
	
	\begin{thm}[A non-homogeneous curvature problem]\label{non-homogeneous-appl}
		Let $\phi>0$ be a smooth function on $(\bbS^{n},g,\nabla)$ with
		\eq{\phi g-\n^{2}\phi\geq 0,}
		$\psi_{\ell}\equiv 1$ and $0< \psi_{k}\in C^{\8}(\bbS^{n})$ for $1\leq k\leq \ell-1$ satisfy \footnote{Note this forces $\psi_{1}$ to be constant.}
		\eq{\n^{2}\psi_{k}-\fr{k}{k+1}\fr{\n\psi_{k}\otimes\n\psi_{k}}{\psi_{k}}+(k-1)\psi_{k}\geq 0.}
		Let $\al$ be a $\Ga_{\ell}$-admissible, Codazzi, non-negative, symmetric $2$-tensor, such that
		\eq{\sum_{k=1}^{\ell}\psi_{k}(x)\si_{k}(\al,g)=\phi(x).\label{non-homogeneous curv prob}}
		Then $\al$ is of constant rank. In particular, when $\al=r_{\bbS^n}[s]\geq 0$ for some positive function $s\in C^{\infty}(\bbS^n)$, then in fact we have $\al>0.$
	\end{thm}

	\begin{proof}
		The result follows quickly from \Cref{CRT}. We define
		\eq{F(\al,g,x) = \sum_{k=1}^{\ell}\psi_{k}(x)\si_{k}(\al,g) - \phi(x).} Since $\si_{k}^{1/k}$ is inverse concave and 1-homogeneous, $F$ is $\Phi$-inverse concave with
		\eq{\Phi^{pq,rs}\eta_{pq}\eta_{rs} := \sum_{k=1}^{\ell}\fr{k+1}{k}\psi_{k}\fr{\si_{k}^{pq}\si_{k}^{rs}}{\si_{k}}\eta_{pq}\eta_{rs}.}
		
		Let $x_0\in M$ and $(e_i)_{1\leq i\leq n}$  be an orthonormal basis of eigenvectors for $\al^{\sharp}(x_0).$
		Now using
		\eq{F^{kr}\al_k^{l}R_{liri}=F^{kr}\al_k^l(g_{lr}g_{ii}-g_{li}g_{ri})=\sum_{k=1}^{\ell} k\psi_k\si_k-F^{ii}\al^{ii},}
		we deduce
		\eq{&\omega_F(\al)(\nabla_{e_i}\al,e_i)+\phi_{;ii}\\
			\geq&\sum_{k=1}^{\ell}\br{\si_{k}\psi_{k;ii}+2\psi_{k;i}\si_{k;i}+\frac{k+1}{k}\frac{\psi_k}{\sigma_k}(\sigma_{k;i})^2+k\psi_{k}\si_{k} }- c\al_{ii}\\
			\geq&\sum_{k=1}^{\ell}\br{\psi_{k;ii}-\fr{k}{k+1}\fr{(\psi_{k;i})^{2}}{\psi_{k}}+(k-1)\psi_{k} + \psi_{k} }\si_{k}-c\al_{ii}\\
			=&\sum_{k=1}^{\ell-1}\br{\psi_{k;ii}-\fr{k}{k+1}\fr{(\psi_{k;i})^{2}}{\psi_{k}}+(k-1)\psi_{k} }\si_{k}+ \phi+ (\ell-1)\sigma_{\ell}- c\al_{ii}.}
		Therefore, $\omega_F(\al)(\nabla_{e_i}\al,e_i)+c\al_{ii}$ is non-negative for some constant $c.$ 
	\end{proof}

	Let $(N,\bar g,\bar{D})$ be a simply connected Riemannian or Lorentzian spaceform of constant sectional curvature $\tau_N$. That is, $N$ is either the Euclidean space $\bbR^{n+1}$, the sphere $\bbS^{n+1}$, the hyperbolic space $\bbH^{n+1}$ with respective sectional curvature $0,1,-1$ or the $(n+1)$-dimensional Lorentzian de Sitter space $\bbS^{n,1}$ with sectional curvature $1$. 
	
	Assume $M=x(\Omega)$ given by $x\cn \Omega\hra N$ is a connected, spacelike, locally convex hypersurface of $N$ and
	\eq{f\in C^{\8}(M\x\bbR_+\x\ti N),}
	where $\ti N$ denotes the {\it{dual manifold}} of $N$, i.e.,
	\eq{\ti \bbR^{n+1}=\bbS^{n},\quad\ti \bbS^{n+1}=\bbS^{n+1},\quad \ti \bbH^{n+1}=\bbS^{n,1},\quad \ti \bbS^{n,1}=\bbH^{n+1}.}
	Here $f$ is extended as a zero homogeneous function to the ambient space. We write $\nu, h, s$ for the future directed (timelike) normal, the second fundamental form and the support function of $M$, respectively  (cf. \cite{BIS6,Bryan2019}). The eigenvalues of $h$ with respect to the induced metric on $\Si$ are ordered as
	$\kappa_1\leq\dots\leq\kappa_n$ and we write in short \[\ka=(\kappa_1,\ldots,\kappa_n).\]
	
	The Gauss equation (cf. \cite[(1.1.37)]{Gerhardt:/2006}) relates extrinsic and intrinsic curvatures,
	\eq{\label{gauss_eq}R_{ijkl}&=\sigma(h_{ik}h_{jl}-h_{il}h_{jk})+\overline{\Rm}(x_{;i},x_{;j},x_{;k},x_{;l}) \\
          &= \sigma(h_{ik}h_{jl}-h_{il}h_{jk})+ \tau_N(\bar{g}_{ik} \bar{g}_{jl} - \bar{g}_{il} \bar{g}_{jk}),}
	where $\si = \bar g(\nu,\nu)$ and the second fundamental form is defined by
	\eq{\bar D_{X}Y = \n_{X}Y - \si h(X,Y)\nu.}

	\begin{thm}\label{thm:hypersurface}
		Let $(N,\bar g,\bar D)$ be one of the spaces above and let $\Psi$ satisfy \Cref{assumption p-CM}. Let $M$ be a connected, spacelike, locally convex and $\Ga$-admissible hypersurface such that
		\eq{\Psi(\ka)=f(x,s,\nu),}
		where $0<f\in C^{\8}(M\x \bbR_{+}\x\ti N)$ and \eq{\bar{D}^2_{xx}f^{-1}+\tau_Nf^{-1} \bar g\geq 0.}
		Then the second fundamental form of $M$ is of constant rank.
	\end{thm}
	
	\begin{proof}
		Define
		$F(h,g,x) = \Psi(h^{\sharp}) - f(x,s(x),\nu(x)).$
		Let $x_0\in M$ and $(e_i)_{1\leq i\leq n}$  be an orthonormal basis of eigenvectors for $h^{\sharp}(x_0).$ Now in  view of \Cref{CRT}, the claim follows from \cite[p.~15]{BIS6} and a computation using the Gauss equation \eqref{gauss_eq}:
		\eq{&\om_F(h)(\n_{e_{m}}h,e_{m}) \\
			\geq~& 2\Psi^{-1}(\Psi_{;m})^2 - \bar D^2_{xx}f(e_m,e_m)+ F^{ik}h^{l}_{i} R_{kmlm} - c(h_{mm} +\abs{\n h_{mm}})\\
			\geq~& 2\Psi^{-1}(\Psi_{;m})^2 - \bar D^2_{xx}f(e_{m},e_{m}) + \Psi^{ik}h^{l}_{i} \bar{R}_{kmlm}- c(h_{mm} + \abs{\n h_{mm}})\\
			\geq~& 2f^{-1}(\bar D_x f(e_{m}))^2 - \bar D^2_{xx}f(e_{m},e_{m}) +\tau_N \Psi^{ik}h^l_i(g_{kl} - g_{km} g_{lm})\\
			&- c(h_{mm}+ \abs{\n h_{mm}})\\
			\geq~& 2f^{-1}(\bar D_x f(e_{m}))^2 - \bar D^2_{xx}f(e_{m},e_{m}) + \tau_N f - c(h_{mm} + \abs{\n h_{mm}})\\
			\geq~& - c(h_{mm} + \abs{\n h_{mm}}).}		
	\end{proof}
	
	The following corollary contains the CRT from \cite{GuanLinMa:/2006,Guan2009c} as special cases.
	
	\begin{cor}[Curvature Measures Type Equations]
		Suppose the curvature function $\Psi$ satisfies \Cref{assumption p-CM}, $1\leq k\leq n-1$, $p\in \bbR$ and $0<\phi\in C^{\infty}(\mathbb{S}^n)$. Let $M$ be a $\Ga$-admissible convex hypersurface of $\bbR^{n+1}$ which encloses the origin in its interior and suppose
		\eq{\Psi(\ka)=\langle x,\nu \rangle^{p} |x|^{-\frac{n+1}{k}}\phi\br{\frac{x}{|x|}}^{\frac{1}{k}}.}
		If \eq{|x|^{\frac{n+1}{k}}\phi\br{\frac{x}{|x|}}^{-\frac{1}{k}}~\mbox{is convex on}~\bbR^{n+1}\setminus\{0\},}
		then $M$ is strictly convex.
	\end{cor}

	\section{A Viscosity Approach}\label{visc}
	
	The following lemma served as the main motivation for us to study the constant rank theorems with a viscosity approach. It shows that the smallest eigenvalue of a bilinear form satisfies a viscosity inequality. In the context of extrinsic curvature flows a similar approach was taken to prove preservation of convex cones; see \cite{Langford14,Langford17}. There it was shown that the distance of the vector of eigenvalues to the boundary of a convex cone satisfies a viscosity inequality. 
	
	\begin{lemma}\label{Brendle's lemma}\cite[Lem. 5]{BrendleChoiDaskalopoulos:/2017} Let the eigenvalues of a symmetric 2-tensor $\al$ with respect to a metric $(g,\nabla)$ at $x_0$ be ordered via
		\eq{\lambda_1=\cdots=\lambda_{D_1}<\lambda_{D_1+1}\leq \cdots\leq \lambda_n,}
		for some $D_1\geq 1.$
		Let $\xi$ be a lower support for $\lambda_1$ at $x_0.$ That is, $\xi$ is a smooth function such that in an open neighborhood of $x_0$,
		\eq{\xi\leq \lambda_1}
		and $\xi(x_0)=\lambda_1(x_0).$ Choose an orthonormal frame for $T_{x_0}M$ such that
		\eq{\alpha_{ij}=\delta_{ij}\lambda_i,\quad g_{ij}=\delta_{ij}.} Then at $x_0$ we have for $1\leq k\leq n$,
		\begin{enumerate}
			\item \eq{\alpha_{ij;k}=\delta_{ij}\xi_{;k}\quad  1\leq i,j\leq D_1,}
			\item \eq{\xi_{;kk}\leq \alpha_{11;kk}-2\sum_{j>D_1}\frac{(\alpha_{1j;k})^2}{\lambda_j-\lambda_1}.}
		\end{enumerate}
	\end{lemma}
	
	While the previous lemma is sufficient for full rank theorems (i.e., when the respective linear map is non-negative, and positive definite at least at one point), we need to generalize \cite[Lem. 5]{BrendleChoiDaskalopoulos:/2017} from the smallest eigenvalue to an arbitrary subtrace of a matrix to treat constant rank theorems.
	
	To formulate the following lemma, we introduce some notation. For a symmetric $2$-tensor $\alpha$ on a vector space $V$ with inner product $g$, let $\alpha^{\sharp}$ be the metric raised endomorphism defined by $g(\alpha^{\sharp}(X), Y) = \alpha(X, Y)$. Then $\alpha^{\sharp}$ is diagonalizable and we write \[\la_1 \leq \dots \leq \la_n\] for the eigenvalues with distinct eigenspaces $E_k$ of dimension $d_k = \dim E_k$, $1 \leq k \leq N$. For convenience, let $E_0 = \{0\}$ and $d_0 = 0$. Define
	\eq{\bar{E}_j = \bigoplus_{k=0}^j E_k, \quad \bar{d}_j = \operatorname{dim} \bar{E}_j}
	for $0\leq j \leq N$ so that
	\eq{\{0\} = \bar{E}_0 \subsetneq \bar{E}_1 \subsetneq \cdots \subsetneq \bar{E}_N = V, \quad \bar{E}_k = \bar{E}_{k-1} \oplus E_k.}
	
	Let $(e_j)_{1 \leq j \leq n}$ be an orthonormal basis of eigenvectors corresponding to the eigenvalues $(\lambda_j)_{1\leq j \leq n}$ giving $E_k = \linspan\{e_{\bar{d}_{k-1}+1}, \dots, e_{\bar{d}_k}\}$ and $\bar{E}_k = \linspan\{e_1, \dots, e_{\bar{d}_k}\}$. For each $1 \leq m \leq n$, there is a unique $j(m)$ such that
	\eq{\bar{E}_{j(m)-1} \subsetneq V_m := \linspan \{e_1, \dots, e_m\} \subseteq \bar{E}_{j(m)}.}
	Then $\bar{d}_{j(m)-1} < m \leq \bar{d}_{j(m)}$. For convenience, we write
	\eq{\label{eq:Dm}\Dm := \bar{d}_{j(m)}.}
	Note that $D_m$ is the largest number such that
	\eq{\lambda_1\leq \cdots\leq \lambda_m=\cdots=\lambda_{D_m}<\lambda_{D_m+1}\leq \cdots\leq \lambda_n,}
	and hence
	\[\bar{E}_{j(m)}=\linspan\{e_1,\ldots,e_{D_m}\}.\]
	The subspace $V_m$ is invariant under $\alpha^{\sharp}$ and the trace of $\alpha^{\sharp}$ restricted to $V_m$ is the subtrace,
	\eq{G_{m}:=\sum_{k=1}^{m}\la_{k}.}
	This subtrace is characterized by Ky Fan's maximum principle (cf. \cite[Thm. 6.5]{MR2325304}), taking the infimum with respect to all traces of $\pi_P \circ \alpha^{\sharp}|_P$ over $m$-planes of the tangent spaces where $\pi_P$ is orthogonal projection onto an $m$-plane $P$:
	\eq{G_{m} &= \inf_P \{\tr \pi_P \circ \alpha^{\sharp}|_P\cn P = m\text{-plane}\} \\
		&= \inf_{(w_{k})_{1\leq k\leq m}}\left\{\sum_{k,l=1}^{m}g^{kl}\al(w_{k},w_{l})\cn (g(w_{k},w_{l}))_{1\leq k,l\leq m}>0\right\},}
	where $(g^{kl})$ is the inverse of $g_{kl}=g(w_{k},w_{l})$. Now suppose $\al$ is a bilinear form on a Riemannian manifold $(M,g)$, $x_0\in M$ and $(e_{i})_{1\leq i\leq n}$ is an orthonormal basis of eigenvectors at $x_{0}$ with eigenvalues
	\eq{\la_{1}(x_{0})\leq \dots\leq \la_{n}(x_{0}).} Letting $w_i(x)$, $1 \leq i \leq m$, be any set of linearly independent local vector fields around $x_0$ with $w_i(x_0) = e_i$, then we have a smooth upper support function for $G_m$ at $x_0$:
	\eq{\Theta(x) := \sum_{k,l=1}^{n} g^{kl} \alpha_{kl} \geq G_m(x), \quad \Theta(x_0) = G_m(x_0),}
	where $\alpha_{kl} = \alpha(w_k(x), w_l(x))$. We make use of $\Theta$ to prove the next lemma generalizing \Cref{Brendle's lemma}.
	
	\begin{lemma}\label{BCD-extension lemma}
		Let $(M,g)$ be a Riemannian manifold and let $\alpha$ be a symmetric $2$-tensor on $TM$. Suppose $1\leq m\leq n$ and  $\xi$ is a (local) lower support at $x_0$ for the subtrace $G_m(\alpha^{\sharp})$. Then at $x_0$ we have
		\begin{enumerate}
			\item \eq{\xi_{;i} = \tr_{V_m} \al_{;i} = \sum_{k=1}^m \al_{kk;i},}
			\item \eq{\xi_{;ii}\leq& \sum_{k=1}^{m}\al_{kk;ii}-2\sum_{k=1}^{m}\sum_{r>D_m}\fr{(\al_{kr;i})^{2}}{\la_{r}-\la_{k}},}
		\end{enumerate}
		where $V_m = \linspan \{e_1(x_0), \dots, e_m(x_0)\}$ for any choice of $m$ orthonormal eigenvectors $e_k$ with corresponding eigenvalues $\lambda_1, \dots, \lambda_m$ satisfying
		\eq{\lambda_1\leq \cdots\leq \lambda_m=\cdots=\lambda_{D_m}<\lambda_{D_m+1}\leq \cdots\leq \lambda_n.}
	\end{lemma}
	
	\pf{For this proof we use the summation convention for indices ranging between $1$ and $m$.
		Let $\xi$ be a lower support for $G_m$ at $x_{0}$. Fix an index $1\leq i\leq n$ and let $\ga(s)$ be a geodesic with $\ga(0)=x_{0}$ and $\dot\ga(0)=e_{i}(x_0)$. Let $(v_k)_{1\leq k \leq m}$ be any basis (not necessarily orthonormal) for $V_m$ as in the statement of the lemma. As mentioned above, for any $m$ linearly independent vector fields $(w_k(s))_{1\leq k \leq m}$  along $\ga$ with $w_k(0) = v_k(x_0)$, $\al_{kl} = \al(w_{k},w_{l})$ and $(g^{kl}) = (g(w_{k},w_{l}))^{-1}$, the function
		\eq{\Theta(s) := g^{kl} \alpha_{kl} -  \xi(\gamma(s))}
		satisfies
		\eq{\Theta(s) \geq 0, \; \Theta(0) = 0}
		and hence
		\eq{\dot{\Theta}(0) = 0, \; \ddot{\Theta}(0) \geq 0.}
		Since $V_m \subseteq \bar{E}_{j(m)}$, choosing $w_k$ such that $\dot{w}_k(0) \perp \bar{E}_{j(m)}(x_0)$ gives
		\eq{\dot{g}_{kl}(0) = g(\dot{w}_k(0), v_l) + g(v_k, \dot{w}_l(0)) = 0}
		and hence also
		\eq{\dot{g}^{kl}(0) = -g^{ka}(0) \dot{g}_{ab}(0)g^{bl}(0) = 0.}
		
		Then we compute
		\eq{
			0 = \dot{\Theta}(0) = \left(g^{kl}  \alpha_{kl;i} - \xi_{;i}\right)|_{x_0}
		}
		giving the first part.
		
		Now we move on to the second derivatives. For this we make the additional assumptions, $v_k = e_k$ and $\ddot{w}_k(0) = 0$. We first calculate
		\eq{\ddot{g}^{kl}(0) &= g^{km}\dot{g}_{mr}g^{ra}\dot{g}_{ab}g^{bl}-g^{ka}\ddot{g}_{ab}g^{bl}+g^{ka}\dot{g}_{ab}g^{bm}\dot{g}_{mr}g^{rl} \\
			&= -\de^{ka}\ddot{g}_{ab}(0)\de^{bl},}
		since $\dot{g}_{kl}(0) = 0$ and $g^{kl}(0) = \delta^{kl}$. Then from $\ddot{w}_k(0) = 0$ we obtain
		\eq{\ddot{g}^{kl}(0) & = -\left[g(\ddot{w}_{k},w_{l})+g(w_{k},\ddot{w}_{l})+2g(\dot{w}_{k},\dot{w}_{l})\right](0) \\
			&= -2 \de^{ka}g(\dot{w}_a(0), \dot{w}_b(0))\de^{bl}.}
		
		From the local minimum property,
		\eq{0&\leq \ddot{\Theta}(0)\\
		&=\ddot{g}^{kl}(0)\al_{kl}+\de^{kl}\fr{d^{2}}{ds^{2}}_{|s=0}\al_{kl}(s)-\xi_{;ii}(x_{0})\\
			&=-2g(\dot{w}_{k}(0),\dot{w}_{l}(0))\al^{kl} + \de^{kl}\al_{kl;ii}\\
			&\hp{=~}+4\de^{kl}\n_{i}\al(\dot{w}_{k}(0),w_{l}(0)) + 2\de^{kl}\al(\dot{w}_{k}(0),\dot{w}_{l}(0))-\xi_{;ii}(x_{0})\\
			&= \sum_{k=1}^m \al_{kk;ii} -\xi_{;ii}(x_{0}) \\
			&\hp{=~} +2\sum_{k=1}^m \br{2\n_{i}\al(\dot{w}_{k}(0),e_{k}) + \al(\dot{w}_{k}(0),\dot{w}_{k}(0)) - g(\dot{w}_{k}(0),\dot{w}_{k}(0))\la_k}.}
		
		From $\dot{w}_k(0) \perp \bar{E}_{j(m)}$, we may write $\dot{w}_k (0) = \sum\limits_{r>D_m} c_k^r e_r$ giving
		\eq{
			\xi_{;ii}(x_{0})-\sum_{k=1}^{m}\al_{kk;ii} &\leq 2\sum_{k=1}^{m} \sum_{r>D_m} \br{2 c_k^r \al_{kr;i} + (c_k^r)^2 \la_r - (c_k^r)^2\la_k }\\
			&= 2\sum_{k=1}^{m} \sum_{r>D_m} c_k^r \left(2\al_{kr;i} + c_k^r(\la_r - \la_k)\right).
		}
		Optimizing yields the specific choice
		\eq{\dot{w}_{k}(0) = -\sum_{r> D_m}\fr{\al_{kr;i}}{\la_{r}-\la_{k}}e_{r}.}
		From this we obtain
		\eq{\xi_{;ii}(x_{0})-\sum_{k=1}^{m}\al_{kk;ii} \leq& -2\sum_{k=1}^{m} \sum_{r>D_m} \fr{\al_{kr;i}}{\la_{r}-\la_{k}} \left(2\al_{kr;i} - \al_{kr;i}\right) \\
			=& -2\sum_{k=1}^{m}\sum_{r>D_m}\fr{(\al_{kr;i})^{2}}{\la_{r}-\la_{k}}.}
	}
	
	\begin{cor}\label{nabla-alpha}
		Let $\al$ be a non-negative, symmetric 2-tensor on $TM$. Suppose for some $1 \leq m \leq n$ that $\dim \ker \al^{\sharp} \geq m-1$ or equivalently that the eigenvalues of $\al^{\sharp}$ satisfy
		$\la_{1} \equiv \dots \equiv \la_{m-1} \equiv 0.$
	Then for all $x_{0}$ and any lower support $\xi$ for $G_{m}$ at $x_{0}$ and all $1\leq i\leq n$ we have
	\begin{enumerate}
	\item $(\n_i\al(x_{0}))_{|\ker \al^{\sharp} \x \ker \al^{\sharp}} = 0,$
	\item	$(\n_i\al(x_{0}))_{|E_{j(m)} \x E_{j(m)}} = g\n_i \xi(x_{0}),\q\mbox{if}~\la_{m}(x_{0})>0.$
    \end{enumerate}
	\end{cor}
	
	\begin{proof}
                We use a basis $(e_{i})$ as in \Cref{BCD-extension lemma}. To prove (1) we may assume  $\lambda_1(x_0) = 0$, and hence the zero function is a lower support for $\lambda_1$. By \Cref{Brendle's lemma}, we have $\nabla \alpha_{kl}=0$ for all $1\leq k,l\leq d_1$ proving the first equation.

Now we prove (2). For $m=1$ the claim follows from \Cref{BCD-extension lemma}-(1). Suppose $m > 1$. If $d_1 \geq m$ at $x_0$ then $\lambda_m(x_0) = 0$ which violates our assumption. Hence $d_1=m-1$ and $E_1(x_0) = \linspan \{e_1, \dots, e_{m-1}\}$. Taking any unit vector $v \in E_2(x_0)=\linspan\{e_m,\ldots,e_{D_m}\}$ and applying \Cref{BCD-extension lemma}-(1) with $V_m=\{e_1,\ldots,e_{m-1},v\}$ gives
		\eq{\nabla_i \alpha(v,v) = \tr_{V_m} \nabla_i\al= \nabla_i \xi\q \forall 1\leq i\leq n.}
		Polarizing the quadratic form $v \mapsto \nabla_i \alpha(v, v)$ over $E_2(x_0)$ then shows \[\nabla_i \alpha_{kl} =\delta_{kl} \nabla_i\xi\quad \fa m \leq k,l \leq \Dm .\]
%
	\end{proof}
	
	Now we state the key outcome of the results in this section. We want to acknowledge that the following proof is inspired by the beautiful paper \cite{Szekelyhidi2016} and their sophisticated test function
	\eq{Q = \sum_{q=1}^{m}G_{q}.}

	\begin{thm}\label{viscosity}
		Under the assumptions of \Cref{CRT}, if  $\dim \ker \alpha^{\sharp} \geq m-1$, for all $\Om\Subset M$ there exists a constant $c=c(\Om)$, such that for all $x_{0}\in \Om$ and any lower support function $\xi$  for $G_m(\alpha^{\sharp})$ at $x_0$ we have
		\eq{F^{ij} \xi_{;ij} \leq c(\xi+ \abs{\nabla \xi}).}
	\end{thm}
	
	\begin{proof}
	In view of our assumption $\la_{m-1}\equiv 0$. Hence the zero function is a smooth lower support at $x_{0}$ for every subtrace $G_{q}$ with $1\leq q\leq m-1$. Therefore by \Cref{BCD-extension lemma}, for every $1\leq q\leq m-1$ and every $1\leq i\leq n$ we obtain
\eq{\label{extra info}0\leq \sum_{k=1}^{q}\al_{kk;ii}-2\sum_{k=1}^{q}\sum_{j>D_{q}}\fr{(\al_{kj;i})^{2}}{\la_{j}-\la_{k}}. }

		Due to the Ricci identity, we have the commutation formula
		\eq{\label{eq:commute}\al_{ij;kl} = \al_{ki;jl} &= \al_{ki;lj} + R^{p}_{kjl}\al_{pi} + R^{p}_{ijl}\al_{pk} \\
			&= \al_{kl;ij} + R^{p}_{kjl}\al_{pi} + R^{p}_{ijl}\al_{pk}.}
		Taking into account  \Cref{BCD-extension lemma} and adding the  inequalities \eqref{extra info} for $1\leq q\leq m-1$, we have at $x_{0}$,
		\eq{F^{ij}\xi_{;ij} &\leq \sum_{q=1}^{m}\sum_{k=1}^{q}F^{ij}\al_{kk;ij}   -2\sum_{q=1}^{m}\sum_{k=1}^{q}\sum_{j>D_{q}}\fr{F^{ii}(\al_{kj;i})^{2}}{\la_{j}-\la_{k}} \\
			&\leq  \sum_{q=1}^{m}\sum_{k=1}^{q}F^{ij}\left(\al_{ij;kk} - R^{p}_{kjk}\al_{pi} - R^{p}_{ijk}\al_{pk}\right)\\			&\hp{=~}-2\sum_{q=1}^{m}\sum_{k=1}^{q}\sum_{j>D_{q}}\fr{F^{ii}(\al_{kj;i})^{2}}{\la_{j}-\la_{k}}.}
		
		Now differentiating the equation
		$F(\al^{\sharp}, x)=0$ yields
		\eq{0&= F^{ij}\al_{ij;k}+D_{x^{k}}F,\\
		0&=F^{ij,rs}\al_{ij;k}\al_{rs;l}+D_{x^{l}}F^{ij}\al_{ij;k}+F^{ij}\al_{ij;kl}+D_{x^{k}}F^{rs}\al_{rs;l}+D_{x^{k}x^{l}}^{2}F.}
		Then substituting above gives
		\eq{F^{ij}\xi_{;ij} &\leq -2\sum_{q=1}^{m}\sum_{k=1}^{q}\sum_{j>D_{q}}\fr{F^{ii}(\al_{kj;i})^{2}}{\la_{j}-\la_{k}} - \sum_{q=1}^{m}\sum_{k=1}^q F^{ij,rs}\al_{ij;k}\al_{rs;k} \\
			&\hp{=~} - \sum_{q=1}^{m}\sum_{k=1}^{q}\br{D^{2}_{x^{k}x^{k}}F + 2D_{x^{k}}F^{ij}\al_{ij;k} + F^{ij}\left(R^{p}_{kjk}\al_{pi} + R^{p}_{ijk}\al_{pk}\right)}\\
			&\leq -2\sum_{q=1}^{m}\sum_{k=1}^{q}\sum_{j>D_{m}}\fr{F^{ii}(\al_{ij;k})^{2}}{\la_{j}} - \sum_{q=1}^{m}\sum_{k=1}^q \sum_{i,j,r,s>D_{m}}F^{ij,rs}\al_{ij;k}\al_{rs;k}\\
			&\hp{=~} - \sum_{q=1}^{m}\sum_{k=1}^{q}\br{D^{2}_{x^{k}x^{k}}F + 2D_{x^{k}}F^{ij}\al_{ij;k} + F^{ij}R^{p}_{kjk}\al_{pi}}+c\xi \\
                        &\hp{=~} + C\sum_{i=1}^{n}\sum_{j,k\leq D_{m}}\abs{\al_{jk;i}}-2\sum_{q=1}^{m}\sum_{k=1}^{q}\sum_{j=D_{q}+1}^{D_{m}}\fr{F^{ii}(\al_{ij;k})^{2}}{\la_{j}},}
where we have used that $\al$ is Codazzi and the fact that $1 \leq k \leq m \leq \Dm$ in splitting the sum involving $F^{ij,rs}$ into terms where at least two indices are at most $D_{m}$ and the remaining indices $i,j,r,s>\Dm$. We have also used $\la_j - \la_k \geq \la_j$, and that for some constant $c,$
		\eq{F^{ij}R^{p}_{ijm}\al_{pm}\geq -c\xi.} 

Now for every $1\leq k\leq m$ define
		\eq{\eta_{k}=(\eta_{ijk}) = \begin{cases}\al_{ij;k},& i,j>D_{m}\\
				0, & i\leq D_{m}~\mbox{or}~j\leq D_{m}.\end{cases}}
				Then 
		\eq{F^{ij}\xi_{;ij} &\leq  -2\sum_{q=1}^{m}\sum_{k=1}^{q}\sum_{j>D_{m}}\fr{F^{ii}(\eta_{ijk})^{2}}{\la_{j}} - \sum_{q=1}^{m}\sum_{k=1}^q F^{ij,rs}\eta_{ijk}\eta_{rsk}\\
			&\hp{=~} - \sum_{q=1}^{m}\sum_{k=1}^{q}D^{2}_{x^{k}x^{k}}F - 2\sum_{q=1}^{m}\sum_{k=1}^{q}D_{x^{k}}F^{ij}\eta_{ijk} - \sum_{q=1}^{m}\sum_{k=1}^{q}F^{ij}R^{p}_{kjk}\al_{pi}\\
			&\hp{=~}+C\sum_{i=1}^{n}\sum_{j,k\leq D_m}\abs{\al_{jk;i}}-2\sum_{q=1}^m\sum_{k=1}^{q}\sum_{j=D_{q}+1}^{D_{m}}\fr{F^{ii}(\al_{ij;k})^{2}}{\la_{j}}+ c\xi.}
		
		In addition we define $\al^{\sharp}_{\ep} = \al^{\sharp}+\ep\id$,
		which has positive eigenvalues for $\varepsilon>0$. In the sequel, a subscript $\ep$ denotes evaluation of a quantity at $\al^{\sharp}_{\ep}$, e.g.,
		we put $F^{ij}_{\ep} = F^{ij}(\al^{\sharp}_{\ep}).$
		We have
		\eq{F^{ij}\xi_{;ij}
		 &\leq \sum_{q=1}^{m}\lim_{\ep\ra 0}\left( - 2\sum_{k=1}^{q}\sum_{j=1}^{n}\fr{F^{ii}_{\ep}(\eta_{ijk})^{2}}{\la_{j}+\ep} - \sum_{k=1}^q F^{ij,rs}_{\ep}\eta_{ijk}\eta_{rsk}\right.\\
			&\hp{=~}\left. - \sum_{k=1}^{q}(D^{2}_{x^{k}x^{k}}F)_{\ep} - 2\sum_{k=1}^{q}(D_{x^{k}}F^{ij})_{\ep}\eta_{ijk} - \sum_{k=1}^{q}F^{ij}_{\ep}R^{p}_{kjk}(\al_{\ep})_{pi}\right )\\
                        &\hp{=~}+C\sum_{i=1}^{n}\sum_{j,k\leq D_m}\abs{\al_{jk;i}}-2\sum_{q=1}^m\sum_{k=1}^{q}\sum_{j=D_{q}+1}^{D_{m}}\fr{F^{ii}(\al_{ij;k})^{2}}{\la_{j}}+c\xi.}
In view of \Cref{Phi def}, and the definition of $\omega_F,$				
		\eq{F^{ij}\xi_{;ij}
			&\leq \sum_{q=1}^{m}\lim_{\ep\ra 0}\left(  - \sum_{k=1}^{q}\Phi^{ij,rs}_{\ep}\eta_{ijk}\eta_{rsk}- \sum_{k=1}^{q}(D^{2}_{x^{k}x^{k}}F)_{\ep}\right.\\
			&\hp{=~}\left.  - 2\sum_{k=1}^{q}(D_{x^{k}}F^{ij})_{\ep}\eta_{ijk} - \sum_{k=1}^{q}F^{ij}_{\ep}R^{p}_{kjk}(\al_{\ep})_{pi}\right )\\
                        &\hp{=~}+C\sum_{i=1}^{n}\sum_{j,k\leq D_m}\abs{\al_{jk;i}}-2\sum_{q=1}^m\sum_{k=1}^{q}\sum_{j=D_{q}+1}^{D_{m}}\fr{F^{ii}(\al_{ij;k})^{2}}{\la_{j}}+c\xi\\
			&\leq - \sum_{q=1}^{m}\sum_{k=1}^{q}\om_{F}(\al)(\eta_{k},e_{k})+C\sum_{i=1}^{n}\sum_{j,k\leq D_m}\abs{\al_{jk;i}}\\
			&\hp{=~}-2\sum_{q=1}^m\sum_{k=1}^{q}\sum_{j=D_{q}+1}^{D_{m}}\fr{F^{ii}(\al_{ij;k})^{2}}{\la_{j}}+c\xi.}
                    Adding and subtracting some terms gives
                      \eq{\label{eq:ineq}F^{ij}\xi_{;ij}
                        &\leq -\sum_{q=1}^{m}\sum_{k=1}^{q}\om_{F}(\al)(\n_{e_{k}}\al,e_{k}) + c\xi\\
                        &\hp{=~}+\sum_{q=1}^{m}\sum_{k=1}^{q}\om_{F}(\al)(\n_{e_{k}}\al,e_{k}) - \sum_{q=1}^{m}\sum_{k=1}^{q}\om_{F}(\al)(\eta_{k},e_{k}) \\
                &\hp{=~}+ C\sum_{i=1}^{n}\sum_{k,j\leq D_m}\abs{\al_{jk;i}}-2\sum_{q=1}^m\sum_{k=1}^{q}\sum_{j=D_{q}+1}^{D_{m}}\fr{F^{ii}(\al_{ij;k})^{2}}{\la_{j}}.}
Next we estimate the last two lines of \eqref{eq:ineq}. We have 
\eq{\sum_{q=1}^{m}\sum_{k=1}^{q}\om_{F}(\al)(\n_{e_{k}}\al,e_{k}) &- \sum_{q=1}^{m}\sum_{k=1}^{q}\om_{F}(\al)(\eta_{k},e_{k})\leq C\sum_{i=1}^{n}\sum_{j,k\leq D_m}\abs{\al_{jk;i}}, \\ C\sum_{i=1}^{n}\sum_{j,k\leq D_m}\abs{\al_{jk;i}}&\leq C\sum_{i=1}^{n}\sum_{k=1}^{D_1}\sum_{j=D_{1}+1}^{D_{m}}\abs{\al_{jk;i}}+c|\nabla \xi|,}
where for the last inequality we used \Cref{nabla-alpha}.
Let us define
\eq{\mathcal{R}&=C\sum_{i=1}^{n}\sum_{k=1}^{D_1}\sum_{j=D_{1}+1}^{D_{m}}\abs{\al_{jk;i}}-2\sum_{q=1}^m\sum_{k=1}^{q}\sum_{j=D_{q}+1}^{D_{m}}\fr{F^{ii}(\al_{ij;k})^{2}}{\la_{j}-\la_{k}}\\
&=C\sum_{i=1}^{n}\sum_{k=1}^{D_1}\sum_{j=D_{1}+1}^{D_{m}}\abs{\al_{jk;i}}-2\sum_{q=1}^{m-1}\sum_{k=1}^{q}\sum_{j=D_{q}+1}^{D_{m}}\fr{F^{ii}(\al_{ij;k})^{2}}{\la_{j}-\la_{k}}.}
Note that if  $\lambda_m(x_0)=0$, then $D_q=D_m$ for all $q\leq m$ and hence $\mathcal{R}=0$.
If $\la_{m}(x_{0})>0,$ then we have $D_q=m-1$ for all $q\leq m-1$ and
\eq{\mathcal{R}=C\sum_{i=1}^{n}\sum_{k=1}^{m-1}\sum_{j=m}^{D_{m}}\abs{\al_{jk;i}}-2\sum_{k=1}^{m-1}(m-k)\sum_{j=m}^{D_{m}}\fr{F^{ii}(\al_{ij;k})^{2}}{\la_{j}-\la_{k}}.}
Therefore, due to uniform ellipticity, we can use
\eq{C\sum_{i=1}^n\abs{\al_{jk;i}}\leq 2(m-k)\fr{F^{ii}(\al_{jk;i})^{2}}{\la_{j}-\lambda_k}+c\xi}
to show that $\mathcal{R}\leq c'\xi.$
Then by the assumptions on $\omega_F$, the right hand side of \eqref{eq:ineq} is bounded by $c(\xi + |\nabla \xi|)$ completing the proof.
                      \end{proof}

                      \begin{rem}
                      Here we crucially used that $F$ is $\Phi$-inverse concave, then we took the limit $\ep \to 0$ and finally swapped $\eta_k$ with $\n_{e_k} \al$ absorbing the extra terms. If on the other hand we tried to swap first without using $\Phi$-inverse concavity, the extra terms would involve $\sum_{r=1}^{n}\fr{F^{ii}_{\ep}(\n_{e_k} (\al_{\ep})_{ir})^{2}}{\la_{r}+\ep}$. Since $\la_r = 0$ for $1\leq r \leq m-1$ this blows up in the limit $\ep \to 0$ and cannot be absorbed.
                      \end{rem}

	

	\begin{proof}[Proof of \Cref{CRT}]
		Let 
		$k := \max_{x\in M}\dim\ker \al^{\sharp}(x).$
		If $k=0$, we are done. 
		By induction we show that for all $1\leq m\leq k$ we have $\la_{m}\equiv 0$.
		For $m=1$, clearly we have $\dim \ker \alpha^{\sharp} \geq m-1$ and hence by \Cref{viscosity} a lower support $\xi$ for $G_1 = \lambda_1$ locally satisfies
		\eq{F^{ij} \xi_{;ij} \leq c(\xi + \abs{\nabla \xi}).}
		By the strong maximum principle \cite{Bardi1999}, $\lambda_1 \equiv 0$. 
		
		Now suppose the claim holds true for $m-1$, i.e.,
		\eq{\la_1 \equiv \dots \equiv \la_{m-1} \equiv 0. }
		Then a lower support $\xi$ for $G_m$ satisfies
		\eq{F^{ij}\xi_{;ij}  \leq c(\xi + \abs{\nabla \xi}).}
		Hence $G_m \equiv 0$ for all $m\leq k.$ Since $k$ indicates the maximum dimension of the kernel, we must have $\la_{k+1}>0$ and  the rank is always $n-k$.
	\end{proof}

	\section*{Acknowledgment}
	PB was supported by the ARC within the research grant ``Analysis of fully non-linear geometric problems and differential equations", number DE180100110. MI was supported by a Jerrold E. Marsden postdoctoral fellowship from the Fields Institute. JS was supported by the ``Deutsche Forschungsgemeinschaft" (DFG, German research foundation) within the research scholarship ``Quermassintegral preserving local curvature flows", grant number SCHE 1879/3-1.
	
	\bibliographystyle{amsalpha-nobysame}

	\vspace{10mm}
	
	\textsc{Department of Mathematics, Macquarie University,\\ NSW 2109, Australia, }\email{\href{mailto:paul.bryan@mq.edu.au}{paul.bryan@mq.edu.au}}
	
	\vspace{2mm}
	
	\textsc{Institut f\"{u}r Diskrete Mathematik und Geometrie,\\ Technische Universit\"{a}t Wien,
		Wiedner Hauptstr 8-10,\\ 1040 Wien, Austria, }\email{\href{mailto:mohammad.ivaki@tuwien.ac.at}{mohammad.ivaki@tuwien.ac.at}}
	
	\vspace{2mm}
	
	\textsc{School of Mathematics, Cardiff University, Senghennydd Road, Cardiff CF24 4AG, Wales, }\email{\href{mailto:scheuerj@cardiff.ac.uk}{scheuerj@cardiff.ac.uk}}

\end{document}